\documentclass[11pt,oneside]{article}
\usepackage[main=english]{babel}
\usepackage[utf8]{inputenc}
\usepackage{amssymb}
\usepackage{amsthm}
\usepackage{amsmath}
\usepackage{nicefrac}
\usepackage{a4wide}
\usepackage{esvect}
\usepackage{hyperref}
\usepackage{cleveref}
\usepackage{thmtools}
\usepackage{thm-restate}
\usepackage{setspace}
\usepackage{verbatim}
\usepackage[dvips]{graphicx}
\usepackage{t1enc}
\usepackage{color}
\usepackage{pgf,tikz}
\usepackage{mathrsfs}
\usepackage{geometry}
\usepackage[noend]{algpseudocode}

\usepackage{enumerate}
\hypersetup{colorlinks=true,linkcolor=blue,citecolor=blue,pdfpagemode=UseNone,pdfstartview={XYZ null null 1.00}}

\newcommand{\ldim}[1]{\operatorname{ldim}\big(#1\big)}
\newcommand{\pdim}[1]{\operatorname{dim}\big(#1\big)}
\newcommand{\twodim}[1]{\operatorname{dim}_2\big(#1\big)}

\newcommand{\cube}[2]{\mathcal{Q}^{#1}_{#2}}

\DeclareMathOperator{\N}{\mathbb{N}}
\DeclareMathOperator{\R}{\mathbb{R}}

\DeclareMathOperator{\Expect}{\mathbb{E}}

\theoremstyle{plain}
\newtheorem{theorem}{Theorem}

\newtheorem{conjecture}{Conjecture}

\newtheorem{proposition}[theorem]{Proposition}

\newtheorem{corollary}[theorem]{Corollary}

\newtheorem{question}{Question}

\begin{document}

\title{The local dimension of suborders of the Boolean lattice}
\author{David Lewis \\{\small Department of Mathematical Sciences} \\ {\small University of Memphis} \\ {\small\tt davidcharleslewis@outlook.com}}

\maketitle

\begin{abstract}
We prove upper and lower bounds on the local dimension of any pair of layers of the Boolean lattice, and show that $\ldim{\cube{n}{1,\lfloor n/2\rfloor}}\sim\frac{n}{\log_2 n}$ as $n\to\infty$. Previously, all that was known was a lower bound of $\Omega(n/\log n)$ and an upper bound of $n$.

Improving a result of Kim, Martin, Masa\v{r}\'{i}k, Shull, Smith, Uzzell, and Wang, we also prove that that the maximum local dimension of an $n$-element poset is at least $\left(\frac{1}{4}-o(1)\right)\frac{n}{\log_2 n}$.


\end{abstract}

\section{Introduction}

Before stating the problems we want to solve, let us review some definitons and notation. The notation we use is mostly standard. Throughout this paper, all posets are assumed to be nonempty (we take this as part of the definition of a poset) and all logarithms are base $2$ unless otherwise specified.


For any $n\in\N$, a chain of cardinality $n$ is denoted by boldface $\mathbf{n}$.
The $n$-dimensional Boolean lattice, defined as the set of all subsets of $[n]$ ordered by inclusion, is denoted $\cube{n}{}$. The suborder of $\cube{n}{}$ induced by the $\ell^\textrm{th}$ and $k^\textrm{th}$ layers (i.e., $\big([n]^{(\ell)}\cup[n]^{(k)},\subseteq\big)$) is called $\cube{n}{\ell,k}$. Because the properties we care about are preserved by poset anti-isomorphisms, we will usually assume that $\ell < k$ and $\ell \leq n/2$.

Let $P = (X,\leq)$ be a poset. A partial linear extension of $P$ is a linear order $L = (Y,\leq_L)$, where $Y \subseteq X$ and, for every $x,y\in Y$, if $x\leq y$, then $x\leq_L y$. A linear extension of $P$ is a partial linear extension whose ground set is $X$. A set $\mathcal{L}$ of partial linear extensions of $P$ is called a local realiser if, for every ordered pair $(x,y)\in P^2$ with $x\not\geq y$, there is an $L\in\mathcal{L}$ such that $x\leq_L y$. A realiser of $P$ is a local realiser whose elements are all linear extensions of $P$. The order dimension of $P$, first introduced by Dushnik and Miller~\cite{dm} and denoted by $\dim{P}$, is the minimum cardinality of a realiser of $P$. Given a local realiser $\mathcal{L}$ and $x\in P$, we write $\mu_\mathcal{L}(x)$ for the number of elements of $\mathcal{L}$ whose ground sets contain $x$, which we call the \emph{multiplicity} of $x$ in $\mathcal{L}$. Then the local dimension of $P$, denoted $\ldim{P}$, is defined as the minimum of $\max\limits_{x\in P}\mu_\mathcal{L}(x)$ over all local realisers $\mathcal{L}$. Local dimension was introduced only recently by Ueckerdt~\cite{ueckerdt}, and is much less well-understood than dimension. Because $\pdim{P}$ is equal to the minimum cardinality of a local realiser of $P$ (see, e.g., Trotter~\cite{trotter}, section 1.12), $\ldim{P} \leq \pdim{P}$ for every poset $P$.

 Kim, Martin, Masa\v{r}\'{i}k, Shull, Smith, Uzzell, and Wang~\cite{localdim} proved that $\ldim{\cube{n}{}}$ is at least $\Omega\big(\frac{n}{\log n}\big)$, but so far the only upper bound we have for $\ldim{\cube{n}{}}$ is the trivial bound of $n$.

Our main result is that the local dimension of the first and middle layers of the Boolean lattice is asymptotically $\frac{n}{\log n}$.

\begin{theorem}\label{thm:main}
As $n\to\infty$, $\ldim{\cube{n}{1,\lfloor n/2\rfloor}} = \frac{n}{\log n} + O\big(\frac{n\log\log n}{(\log n)^2}\big)$.
\end{theorem}

Another dimension variant, called $t$-dimension, was introduced by Nov{\'a}k~\cite{novak}. For a poset $P$ and $t\in\N$ with $t\geq 2$, the $t$-dimension of $P$, denoted $\dim_t(P)$, is the smallest cardinal $d$ such that $P$ embeds into a product of $d$ chains of cardinality $t$. The most interesting case is $\twodim{P}$, which is the smallest $d$ such that $P$ embeds into $\cube{d}{}$ as a suborder. For example, Sperner's theorem states that, if $A_n$ is an antichain of size $n$, then $\twodim{A_n} = \min\left\{m : \binom{m}{\lfloor m/2\rfloor} \geq n\right\} = \log n + \frac{1}{2}\log\log n + O(1)$.
We also have $\twodim{\mathbf{n}} = n-1$. Clearly $\twodim{P} \leq |P|$ for every poset $P$ (send each $a\in P$ to $\{x\leq a\}$), and this bound is sharp for $n\geq 2$ (see exercise 10.2.6 in \cite{trotter}). Also, by the pigeonhole principle, $\twodim{P} \geq \lceil\log |P|\rceil$, and this bound is also sharp.

\section{Lexicographic sums}

Let $P$ be a poset with ground set $X$ and, for each $x\in X$, let $Q_x$ be a poset with ground set $Y_x$. The lexicographic sum of $\{Q_x\}$ over $P$, denoted $\sum\limits_{x\in P} Q_x$, is a poset on the ground set $\{(x,y) : x\in X, y\in Y_x\}$ where $(x,y) \leq (z,w)$ if and only if either $x < z$ or $x=z$ and $y\leq w$. Hiraguchi~\cite{hiraguchi} proved that
\[\pdim{\sum\limits_{x\in X} Q_x} = \max\big\{\pdim{P},\max\{\pdim{Q_x} : x\in X\}\big\}.\]
We don't have such a simple equation for local dimension, but we can prove some weaker inequalities.

\begin{proposition}\label{prop:lex}
For any poset $P$ with ground set $X$ and any family $\{Q_x\}_{x\in X}$ of nonempty posets indexed by $X$, we have the following inequalities:
\begin{align}
\label{ineq:lex1}
\ldim{\sum\limits_{x\in X} Q_x}\geq \max\big\{\ldim{P},\max\{\ldim{Q_x} : x\in X\}\big\},\\
\label{ineq:lex2}
\ldim{\sum\limits_{x\in X} Q_x} \leq
\max\big\{\ldim{P},\max\{\pdim{Q_x} : x\in X\}\big\},\\
\label{ineq:lex3}
\ldim{\sum\limits_{x\in X} Q_x} \leq \ldim{P} + \max\left\{\ldim{Q_x} : x\in X\right\}
.
\end{align}
\end{proposition}
\begin{proof}
Inequality~\ref{ineq:lex1} follows trivially from the fact that $\sum\limits_{x\in P} Q_x$ has suborders isomorphic to $P$ and to $Q_x$ for each $x\in X$.

To prove inequality~\ref{ineq:lex2}, let $\mathcal{L}$ be a local realiser of $P$ and, for each $x\in X$, let $\mathcal{M}_x$ be a realiser of $Q_x$. For convenience, we regard partial linear orders as lists rather than posets. We will construct a local realiser of $\sum\limits_{x\in P} Q_x$ as follows. For each $x\in X$, if $|\mathcal{M}_x| \leq \mu_\mathcal{L}(x)$, replace each occurrence of $x$ in the lists in $\mathcal{L}$ with $x\times M$ for some $M\in\mathcal{M}_x$, using each list in $\mathcal{M}_x$ at least once. If $\mu_\mathcal{L}(x) < |\mathcal{M}_x|$, replace each occurrence of $x$ in the lists in $\mathcal{L}$ with $x\times M$ for some $M\in\mathcal{M}_x$, using a different list $M$ for each occurrence. Then, for each unused $M\in\mathcal{M}_x$, add $x\times M$ as a new list. Let $\mathcal{N}$ be the set of all lists thus constructed. For each $x\in X$ and each $y\in Y_x$, we can see that $\mu_\mathcal{N}(x,y) = \max\left\{\mu_\mathcal{L}(x),|\mathcal{M}_x|\right\}$ by counting the number of occurrences of $(x,y)$ in these lists. To show that $\mathcal{N}$ is a local realiser, suppose $(x,y) \not\geq (x',y')$. Then either $x=x'$ and $y\not\geq y'$ or $x\not\geq x'$. If the former, then there is some $M\in\mathcal{M}_x$ such that $y$ occurs before $y'$ in $M$, and $x\times M$ is a sublist of some list in $\mathcal{N}$. If the latter, then there is some $L\in\mathcal{L}$ in which $x$ occurs before $x'$. In the corresponding element of $\mathcal{N}$, $x$ and $x'$ have been replaced by sublists containing $(x,y)$ and $(x',y')$ respectively, so $(x,y)$  occurs before $(x',y')$ in this list.

To prove inequality~\ref{ineq:lex3}, let $\mathcal{L}$ be a local realiser of $P$ and, for each $x\in X$, let $\mathcal{M}_x$ be a local realiser of $Q_x$. For each $x\in X$, let $K_x$ be an arbitrary linear extension of $Q_x$. Let $\mathcal{N}$ be the set of all lists obtained by replacing, for each $x\in X$, each occurrence of $x$ in the linear orders in $\mathcal{L}$ with $x\times K_x$ as well as all lists of the form $x\times M$ with $x\in X$ and $M\in\mathcal{M}_x$. This set is a local realiser of $\sum\limits_{x\in X} Q_x$, and, for every $x\in X$ and $y\in Y_x$, $\mu_\mathcal{N}(x,y) = \mu_\mathcal{L}(x) + \mu_{\mathcal{M}_x}(y)$. The special case where $P$ is an antichain was stated as an exercise by Bosek, Grytczuk, and Trotter in \cite{planar}.
\end{proof}

\section{Lower bounds}

Wang Zhiyu~\cite{zhiyu}, following a comment by Christophe Crespelle in \cite{localdim}, suggested that an information entropy method might help improve the bounds on $\ldim{\cube{n}{}}$.

Let $X$ be a discrete random variable taking values in $\{x_1,x_2,\dots,x_n\}$ with $\mathbb{P}\{X = x_i\} = p_i$ for each $i$. Then the entropy of $X$, denoted $H(X)$, is defined by the formula
\[H(X) = -\sum\limits_{\substack{1\leq i\leq n \\ p_i\neq 0}}p_i\log p_i.\]
It can be shown using the strict concavity of $\log$ that $H(X) \leq \log n$, and that this bound is only attained when $X$ is uniformly distributed. Entropy was introduced by Shannon~\cite{shannon}, and $H(X)$ can be thought of as the average amount of information in bits obtained be observing the value of $X$. Shannon's fundamental theorem for a noiseless channel, also known as the source coding theorem, makes this precise. Before stating the theorem, we need a few definitions. An \emph{alphabet} is a finite set with two or more elements, and the elements of an alphabet are called symbols. Given an alphabet $\Omega$, a \emph{word} over $\Omega$ is a finite sequence of symbols in $\Omega$, and $\Omega^\star$ denotes the set of all words over $\Omega$. Given a finite set $\Sigma$ and an alphabet $\Omega$, a \emph{prefix-free code} is a map $C$ from $\Sigma$ to $\Omega^\star$ such that, for all $x,y\in\Sigma$, if $x\neq y$, then $C(x)$ is not an initial segment of $C(y)$.

\begin{theorem}
\label{thm:shannon}
Let $X$ be a random variable taking values in a finite set $\Sigma$ and let $\Omega$ be a finite alphabet. For every prefix-free code $C$, the expected length of $C(X)$ is at least $\frac{H(X)}{\log|\Omega|}$. Conversely, there exists a prefix-free code $C$ such that the expected length of $C(X)$ is at most $\frac{H(X)}{\log|\Omega|}+1$.\hfill\qedsymbol
\end{theorem}


In a note added to the end of \cite{localdim}, Crespelle suggested a method of encoding an arbitrary poset on a fixed ground set, which we call the Crespelle code.
Given a ground set $X$ with cardinality $n$ and a poset $P$ on $X$, we encode $P$ as a word over a $3n$-symbol alphabet $\Omega$ as follows. Let $\Omega = \{x_i,x_m,x_f:x\in X\}$. Given a local realiser $\mathcal{L}$ of $P$, we write each nontrivial\footnote{Of course, if we remove all the one-element partial extensions from a local realiser, the resulting set is still a local realiser.} partial extension in $\mathcal{L}$ as a list of elements of $X$, using symbols of the form $x_i$ at the beginning of each list and symbols of the form $x_m$ everywhere else. Then we concatenate these lists in any order and replace the last symbol $x_m$ with the corresponding $x_f$. We call the resulting word a \emph{Crespelle codeword} for $P$. Note that this code is prefix-free by construction. If $\ldim{P} = d$, then $P$ has a local realiser $\mathcal{L}$ for which each element of $X$ has multiplicity at most $d$, and the Crespelle codeword for $P$ constructed from $\mathcal{L}$ has length at most $dn$.


Kleitman and Rothschild~\cite{kleitmanrothschild} proved that the entropy of a uniformly random partial order on $[n]$ is $(\frac{1}{4}+o(1))n^2$. Kim, Martin, Masa\v{r}\'{i}k, Shull, Smith, Uzzell, and Wang~\cite{localdim} proved that the maximum local dimension of a poset on $n$ points is at between $\left(\frac{1}{4e\ln 2}-o(1)\right)\frac{n}{\log n}$ and $(4+o(1))\frac{n}{\log n}$ $n\to\infty$, so every partial order on $[n]$ has a Crespelle codeword with at most $(4+o(1))n^2$ bits. As Crespelle observed, this means that the Crespelle code is optimal up to a constant factor.

We can use this fact improve the lower bound on the maximum local dimension of an $n$-element poset.

\begin{theorem}\label{thm:avgldim}
As $n\to\infty$, the expected local dimension of a poset chosen uniformly at random from the set of all $n$-element labelled posets is at least
\[\left(\frac{1}{4}-o(1)\right)\frac{n}{\log n}.\]
\end{theorem}
\begin{proof}
Suppose $n\geq 2$ and let $P$ be a random partial order on $[n]$, where we assign equal probability to each partial order. Assume for convenience that $n$ is even. The total number of such partial orders is at least the number of two-layer partial orders with minimum elements $1,2,3,\dots, n/2$ and maximal elements $n/2+1, n/2+2,\dots,n$, which is equal to $2^{\frac{1}{4}n^2}$. It follows that $H(P) \geq \frac{1}{4}n^2$. It is not much harder to show that this is also true when $n$ is odd.

Let $d = \Expect[\ldim{P}]$.
Then the expected length of the shortest Crespelle codeword for $P$ is at most $dn$. Hence by Theorem~\ref{thm:shannon} $H(P) \leq dn\log 3n$, so
\[
d \geq \frac{n}{4\log 3n}.
\]
\end{proof}

A similar argument shows that the expected $2$-dimension of a random partial order on $[n]$ is at least $\frac{1}{4}n-o(1)$ as $n\to\infty$. First, we define a prefix-free binary code for the set of all partial orders on $[n]$ as follows. Suppose $P$ is a partial order on $[n]$ with $2$-dimension $d$. Fix a poset embedding $f$ from $P$ into $\cube{d}{}$. The codeword for $P$ consists of a block of $\lceil\log n\rceil$ bits representing $d$ as a binary number, followed by $n$ blocks of $d$ bits each, where the $i^\textrm{th}$ block is the representation of $f(i)$ as a binary string. The length of this word is $\lceil\log n\rceil + dn$ bits. Now let $P$ be a uniformly random partial order on $[n]$, as defined in Theorem~\ref{thm:avgldim}. Repeating the proof of Theorem~\ref{thm:avgldim} using this binary code instead of the Crespelle code, we find that $\Expect[\twodim{P}] \geq \frac{1}{4}n-\frac{\lceil\log n\rceil}{n}$ for all $n\geq 2$. By essentially the same argument, we have $\Expect[\dim_t\left(P\right)] \geq \frac{1}{4\log t}n - \frac{\lceil\log_t n\rceil}{n}$ for every $t \geq 2$ and $n\geq 2$.

It's clear from the proofs that the same results hold (up to an additive $o(1)$ term) for a uniformly random two-level poset with minimum elements $1,2,\dots,\lfloor n/2\rfloor$ and maximum elements $\lfloor n/2\rfloor+1, \lfloor n/2\rfloor+2,\dots, n$.

The following lower bound applies to any pair of layers in the Boolean lattice.
\begin{theorem}\label{thm:lkldim}
For any $\ell,k < n$,
\[
\ldim{\cube{n}{\ell,k}} \geq \frac{\log\binom{n}{k}}{\log\binom{n}{\ell}} - \frac{\log\binom{n}{k}}{\big(\log\binom{n}{\ell}\big)^2}\left(\log\log\binom{n}{\ell}+c\right),
\]
where $c \leq \log 12 < 3.585$. In particular,
\[
\ldim{\cube{n}{1,k}} \geq \left(\frac{1}{\log n} - O\left(\frac{\log\log n}{(\log n)^2}\right)\right)\log\binom{n}{k}.
\]
\end{theorem}

\begin{proof}
Assume $\ell < k$ and let $P$ be a random two-level poset defined as follows. Let $A = [n]^{(\ell)}$ and $B = [m]$. For each $b\in B$, choose a random subset $X_b$ of $[n]$ of cardinality $k$, and then let $P$ be the two-level poset on $A\dot{\cup} B$ where each $b\in B$ is above $a\in A$ if $a\subseteq X_b$.
Because the entropy of the joint distribution of mutually independent random variables is the sum of the entropies of those random variables, $H(P)$ is at least $m\log\binom{n}{k}$.
Now define $d = \mathbb{E}[\ldim{P}]$. The expected length of the shortest Crespelle codeword for $P$ is at most $d \left(\binom{n}{\ell}+m\right)$. Hence by Theorem~\ref{thm:shannon} $H(P) \leq d\left(\binom{n}{\ell}+m\right)\left(\log\left(3\binom{n}{\ell}+3m\right)\right)$, so
\[
d \geq \frac{m\log\binom{n}{k}}{\left(\binom{n}{\ell}+m\right)\left(\log\left(3\binom{n}{\ell}+3m\right)\right)}
.
\]
If we set $m = \left\lfloor\binom{n}{\ell}\left(\log\binom{n}{\ell}-1\right)\right\rfloor$, then
\[
d \geq \frac{\log\binom{n}{k}}{\log\binom{n}{\ell}} - \frac{\log\binom{n}{k}}{\big(\log\binom{n}{\ell}\big)^2}\left(\log\log\binom{n}{\ell}+\log 6+\binom{n}{\ell}^{-1}\right).
\]
Now assume $\ldim{P}\geq d$ (which occurs with nonzero probability) and modify $P$ as follows to obtain a new poset $P'$. For each subset $S$ of $A$, if there is more than one vertex in $B$ whose neighbourhood is $S$, delete all but one of them. Since $P$ is a lexicographic sum of antichains over $P'$ and $P'$ is not a chain, by Proposition~\ref{prop:lex}, $P$ and $P'$ have the same local dimension. Since $P'$ embeds into $\cube{n}{\ell,k}$, $\ldim{\cube{n}{\ell,k}} \geq \ldim{P'} \geq d$.
A similar argument works when $k<\ell$.
\end{proof}

This has the following immediate corollary.

\begin{corollary}\label{cor:alpha}
For any $\alpha \in [0,1]$, as $n\to\infty$, $\ldim{\cube{n}{1,\lfloor n^\alpha\rfloor}} \geq (1-\alpha)n^\alpha - O\big(\frac{n^\alpha\log\log n}{(\log n)^2}\big)$, and the same is true for $\cube{n}{1,(n-\lfloor n^\alpha\rfloor)}$.

\hfill\qedsymbol
\end{corollary}


We also have the following lower bound for the first and middle layers. The bound $\ldim{\cube{n}{}} \geq (1+o(1))\frac{n}{\log n}$, which improves one of Kim et al.'s results by a constant factor, was also proved by Stefan Felsner by different means. Felsner's proof can be found in \cite{diffgraph}.

\begin{corollary}
As $n\to\infty$, $\ldim{\cube{n}{}} \geq \ldim{\cube{n}{1,\lfloor n/2\rfloor}} \geq \frac{n}{\log n} - O\big(\frac{n\log\log n}{(\log n)^2}\big)$.\hfill\qedsymbol
\end{corollary}

The following lower bound for the local dimension of $\cube{n}{1,2}$ was proved by Gir\~ao and Lewis, and the proof was included in~\cite{diffgraph}.

\begin{theorem}\label{thm:12}
As $n\to\infty$, $\ldim{\cube{n}{1,2}} \geq \log\log n - O(\log\log\log n)$.
\hfill\qedsymbol
\end{theorem}
Spencer~\cite{spencer} proved that $\pdim{\cube{n}{1,2}} \leq \log\log n + O(\log\log\log n)$, so this bound is asymptotically the best possible.
Because $\cube{n-k+1}{1,2}$ embeds into $\cube{n}{k,k+1}$ and $\cube{n}{\ell,k}$ embeds into $\cube{n+1}{\ell,k+1}$, we also have
\[\ldim{\cube{n}{\ell,k}} \geq (1-o_{\ell,k}(1))\log\log n\] as $n\to\infty$, for every fixed $\ell<k\in\N$.

One immediate corollary of Theorem~\ref{thm:12} is that $\R^n$ has local dimension $(1-o(1))n$. In fact, this implies that $\ldim{\R^n} = n$, because otherwise we would get a violation of subadditivity. It then follows by a compactness argument that, for every $n\in\N$, there exists a finite poset with dimension and local dimension $n$. Barrera-Cruz, Prag, Smith, Taylor, and Trotter~\cite{bpstt} proved this result independently using the Product Ramsey Theorem.

\subsection{Suborders of the divisibility lattice}
In \cite{lewissouza}, Lewis and Souza studied the dimension of suborders of the divisibility lattice $(\N,\mid)$. They proved that
\[\ldim{([n],\mid)} \leq \pdim{([n],\mid)} \leq \left(4\ln 2+o(1)\right)\frac{(\log n)^2}{\log\log n}.\]
We can prove the following lower bound using Corollary~\ref{cor:alpha}.
\begin{proposition}
As $n\to\infty$,
\[\ldim{([n],\mid)} \geq \left(\frac{1}{4}-o(1)\right)\frac{\log n}{\log\log n}.\]
\end{proposition}
\begin{proof}
Fix $c < 1$ and let $k = \big\lfloor\frac{c^2}{4}\big(\frac{\log n}{\log\log n}\big)^2\big\rfloor$. Because $p_k^{\sqrt{k}} = (\log n)^{(c+o(1))\frac{\log n}{\log\log n}} = n^{c+o(1)} \ll n$, we can define an embedding from $\cube{k}{1,\lfloor\sqrt{k}\rfloor}$ into $([n],\mid)$, which therefore has local dimension at least $(\frac{c}{4}-o(1))\frac{\log n}{\log\log n}$. Now taking $c \to 1$ as slowly as necessary completes the proof of the lower bound.
\end{proof}

\section{Upper bounds for two layers}



In this section, we prove some upper bounds on the local dimension of posets of the form $\cube{n}{\ell,k}$.

The following proposition gives good upper bounds on the local dimension of suborders of the form $\cube{n}{\ell,(n-k)}$ when $\ell$ and $k$ are constant.
\begin{proposition}
Whenever $\ell < n-k$, $\ldim{\cube{n}{\ell,(n-k)}} \leq 2 + \max\{\ell,k\}$.
\end{proposition}
\begin{proof}
Let $\pi_0$ be the linear extension consisting of $\ell$-sets in any order, followed by the $(n-k)$-sets in any order. Then let $\pi_1$ be the $\ell$-element sets in the opposite order as in $\pi_0$ followed by the $(n-k)$-element sets in the opposite order as in $\pi_0$. Then, for each $i\in[n]$, let $L_i$ be the $(n-k)$-element sets not containing $i$ (in any order) followed by the $\ell$-element sets containing $i$. It's obvious that $\{\pi_0,\pi_1,L_1,L_2,\dots,L_n\}$ is a local realiser of $\cube{n}{\ell,(n-k)}$ in which every $\ell$-set has multiplicity $2+\ell$ and every $(n-k)$-set has multiplicity $2+k$.
\end{proof}
This generalises Ueckerdt's~\cite{ueckerdt} result that $\ldim{S_n} = 3$ for $n \geq 3$. By contrast, F\"{u}redi~\cite{furedi} proved that, for any constant $k\geq 3$, $\pdim{\cube{n}{k,(n-k)}} = n-2$ for all sufficiently large $n$.

\begin{corollary}
For any fixed $\alpha\in(0,1)$, $n^\alpha\leq \ldim{\cube{n}{1,(n-\lfloor n^\alpha\rfloor)}} = \Theta(n^\alpha)$, with constants between $1-\alpha$ and $1$.
\end{corollary}

For pairs of layers that are close together, we have the following upper bound due to Brightwell, Kierstead, Kostochka, and Trotter~\cite{mid}.
\begin{theorem}
For any $s,k,n\in\N$ with $s+k \leq n$, $\ldim{\cube{n}{s,(s+k)}} \leq \pdim{\cube{n}{s,(s+k)}} \leq (4k^2+18k)\lceil\ln n\rceil$. In the case $k=1$, $\ldim{\cube{n}{s,(s+1)}} \leq \pdim{\cube{n}{s,(s+1)}} \leq 6\lceil\log_3 n\rceil$.
\hfill\qedsymbol
\end{theorem}

Kostochka~\cite{kostochka} later improved the second bound.
\begin{theorem}
For all $s \leq n$
\[
\ldim{\cube{n}{s,s+1}}\leq\pdim{\cube{n}{s,s+1}} \leq 2\min\big\{k:2\cdot k!\geq n\big\} = (2+o(1))\frac{\log n}{\log\log n}.
\].\hfill\qedsymbol
\end{theorem}

Dushnik~\cite{dushnik} showed that $\pdim{\cube{n}{1,k}} \geq n - \sqrt{n}$ whenever $k \geq 2\sqrt{n}$. The next theorem shows that local dimension behaves very differently.
\begin{theorem}\label{thm:upperbound}
For any $n\in\N$ and $k\leq n$,
\[\ldim{\cube{n}{1,k}} \leq \frac{n}{\log n} + \frac{2n\log\log n}{(\log n)^2}+3.\]
More generally, whenever $1 \leq \ell<k\leq n$ and $\ell < \frac{n}{\log n}$, $\ldim{\cube{n}{\ell,k}} \leq (1+o_\ell(1))\frac{n}{\log n}$.
\end{theorem}
\begin{proof}


We first describe a general construction, then show two different ways the construction can be realised. Let $G$ be an $n$-edge bipartite graph with parts $A$ and $B$ and suppose each vertex in $A$ has degree at most $\Delta$. We identify $\cube{n}{\ell,k}$ with the set of all $k$-subsets and $\ell$-subsets of $E(G)$ and define a local realiser of $\cube{n}{\ell,k}$ as follows. For each $v\in A$ and each $X\subseteq \Gamma(v)$, let $L_{v,X}$ be a partial linear extension that lists all $k$-sets $S$ such that $\{u\in B : vu\in S\} = X$ followed by all $\ell$-sets containing an edge $vu$ with $u\not\in X$. Now let $\pi_0$ list all the $\ell$-sets of edges in some order followed by all the $k$-sets, and let $\pi_1$ list all the $\ell$-sets in the opposite order followed by all the $k$-sets in the opposite order. Then $\big\{\pi_0,\pi_1\big\}\cup\big\{L_{v,X} : v\in [A], X\subset [B]\big\}$ is a local realiser of $\cube{n}{\ell,k}$ in which every $\ell$-set has multiplicity at most $2^{\Delta-1}\ell+2$ and every $k$-set has multiplicity at most $|A|+2$.

For any $n$, $\ell$, and $k$ with $1\leq \ell<k\leq n$ and $\ell < \frac{n}{\log n}$, we may take $G$ to be any $n$-edge bipartite graph with $|A| = \big\lceil\frac{n}{\log n - \log\log n-\log\ell}\big\rceil$ and $|B| = \lceil\log n - \log\log n-\log\ell\rceil$. Here we have $2^{\Delta-1}+2\ell \leq 2^{|B|-1}\ell+2 < \frac{n}{\log n}+2$ and $|A|+2 \leq \frac{n}{\log n} + \frac{2n}{(\log n)^2}(\log\log n + \log \ell) + 3 $, which gives us an upper bound of $\ldim{\cube{n}{\ell,k}}\leq\frac{n}{\log n} + \frac{2n}{(\log n)^2}(\log\log n + \log \ell) + 3$.

We can do slightly better when $n = 2^{m-1}m\ell$ for some $m\in N$. In this case, we may take $G$ to be the disjoint union of $\ell$ copies of the $m$-dimensional hypercube graph. For this graph, $\Delta = m$ and $|A| = 2^{m-1}\ell$, so we have $\ldim{\cube{n}{\ell,k}} \leq 2^{m-1}\ell + 2 = \frac{n}{m}+2$. Since $m = \log\frac{n}{\ell} - \log m +1$, this bound is asymptotically the same as the one above.

\end{proof}

Theorem~\ref{thm:main} follows immediately from Theorem~\ref{thm:lkldim} and Theorem~\ref{thm:upperbound}.

\section{Problems}

It's known (see, for example, Dushnik~\cite{dushnik}) that $\pdim{\cube{n}{1,k}}$ is monotone in $k$ for $0 \leq k \leq n-1$, but $\ldim{\cube{n}{1,k}}$ is not, as $\ldim{\cube{n}{1,n/c}} \geq \Omega_c(n/\log n)$ while $\ldim{\cube{n}{1,n-k}} \leq k+2$. We do not know whether or not it's unimodal. Our best upper and lower bounds are unimodal, with a single peak at $k\approx n/2$ -- although the upper bound is more like a plateau.

\begin{question}
How does the function $f_n(k) := \ldim{\cube{n}{1,k}}$ behave? Is it unimodal?
\end{question}

A poset has a short Crespelle codeword if and only if it has a local realiser with small average multiplicity. It follows that a poset with small local dimension has a short Crespelle codeword. Is the converse true? In other words, how far apart can the smallest possible average multiplicity be from the smallest possible maximum multiplicity of a local realiser? Of course, this question is trivial, since adding a top element to a nontrivial poset does not change its local dimension. By repeatedly adding new top elements to a $d$-local-dimensional poset, we obtain $d$-local-dimensional posets with local realisers that have average multiplicity arbitrarily close to $1$. The question becomes nontrivial (or at least not obviously trivial) if we consider both cardinality and local dimension.

\begin{question}
For $d,n\in \N$, what is the minimal encoding length of a poset with local dimension $d$ and cardinality $n$?
\end{question}

The dimension of a lexicographic sum is determined by the dimensions of the summands and of the indexing poset, but it's not clear whether or not the same is true of local dimension. As Bosek, Gryczuk, and Trotter stated in \cite{planar}, it's an open question whether or not Inequality~\ref{ineq:lex3} in Proposition~\ref{prop:lex} can be improved, even when the indexing poset is an antichain.
\begin{question}
Can any of the bounds in Proposition~\ref{prop:lex} be improved?
\end{question}

Kim et al.~\cite{localdim} asked whether or not $\ldim{\cube{n}{}} = n$ for all $n$. We believe that this is not the case, and that our best lower bound is asymptotically correct.

\begin{conjecture}\label{conj:weakcube}
As $n\to\infty$, $\ldim{\cube{n}{}} = \Theta\big(\frac{n}{\log n}\big)$.
\end{conjecture}
If this conjecture is true, it would imply that $\ldim{P} = O\left(\twodim{P} / \log \twodim{P}\right)$ for every poset $P$, and hence provide a new proof of Kim et al.'s theorem that $\ldim{P} = O\left(|P| / \log |P|\right)$ for every $P$, perhaps even improving it by a constant factor.

We may even propose the following stronger conjecture.
\begin{conjecture}\label{conj:strongcube}
There exists a universal constant $c$ such that, for any $t\in\N$, as $n\to\infty$, $\ldim{\mathbf{t}^n} \sim c\frac{n}{\log_t n}$.
\end{conjecture}

One can prove a lower bound of this form using an entropy argument similar to the proof of Theorem~\ref{thm:lkldim}. Indeed, suppose $n \ll m=n^{1+o(1)}$. Let $P$ be a random partial order on $[(t-1)n+m]$ induced by a random function from $[(t-1)n+m]$ to $\mathbf{t}^n$ that maps $[(t-1)n]$ to the set of $n$-tuples with exactly one nonzero entry in lexicographic order. Each such function induces a different labelled poset. Let $d = \Expect[\ldim{P}]$. Then $H(P) = (1+o(1))mn\log t \leq (m+(t-1)n)d\log\left(m+(t-1)n\right) = (1+o(1))md\log n$, so $d \geq (1+o(1))\frac{n}{\log_t n}$.

Even finding one cube with local dimension smaller than its dimension would be very good. If $\ldim{\cube{k}{}} = \ell < k$, then, by subadditivity, $\ldim{\cube{n}{}} \leq \frac{\ell}{k}n + \ell$ for all $n$.

Searching for good local realisers by computer does not seem feasible. Proving or disproving these conjectures will require new ideas.



\bibliography{ldimboolean_new}{}
\bibliographystyle{plain}
\end{document}